\documentclass[preprint,12pt]{elsarticle}

\usepackage{amssymb}
\usepackage{amssymb}
\usepackage{graphicx}
\usepackage{epsfig}
\usepackage{epstopdf}
\usepackage{amsmath}
\usepackage{placeins}
\usepackage{amsfonts,amsmath, amsthm}
\usepackage{verbatim}
\usepackage{xfrac}
\usepackage[all]{xy}
\newtheorem{theorem}{Theorem}[section]
\newtheorem{lemma}[theorem]{Lemma}
\newtheorem{corollary}[theorem]{Corollary}
\newtheorem{pro}[theorem]{Proposition}
\newtheorem{statement}[theorem]{Statement}

\newtheorem{example}[theorem]{Example}

\theoremstyle{remark}
\newtheorem{rem}[theorem]{Remark}

\numberwithin{equation}{section}

\DeclareMathOperator{\So}{SO}
\DeclareMathOperator{\vol}{vol}

\DeclareMathOperator{\Sl}{SL}
\DeclareMathOperator{\id}{id}

\DeclareMathOperator{\Gl}{GL}
\DeclareMathOperator{\pr}{pr}

\begin{document}



\begin{frontmatter}
\title{Existence of an attractor for a geometric tetrahedron transformation}

\author{Dimitris Vartziotis}
\address{NIKI Ltd. Digital Engineering, Research Center, 205 Ethnikis Antistasis Street, 45500 Katsika, Ioannina, Greece}
%\and \\TWT GmbH Science \& Innovation, Department for Mathematical Research \&
%Services, Ernsthaldenstr. 17, 70565 Stuttgart, Germany }
\author{Doris Bohnet}
\address{TWT GmbH Science \& Innovation, Department for Mathematical Research \&
Services, Ernsthaldenstr. 17, 70565 Stuttgart, Germany}

\begin{abstract}
We analyze the dynamical properties of a tetrahedron transformation on the space of non-degenerate tetrahedra which can be identified with the non-compact globally symmetric $8$-dimensional space $\Sl(3,\mathbb{R})/\So(3,\mathbb{R})$. We establish the existence of a local attractor which coincides with the set of regular tetrahedra and identify conditions which imply that the basin of attraction is the entire space. In numerical tests, these conditions are fulfilled for a large set of random tetrahedra.
\end{abstract}

\begin{keyword}
geometric transformation, symmetry group, equivariant dynamical system, attractor, homogeneous space, Lyapunov exponents

\MSC 37L30, 43A85, 51N30, 22E15
\end{keyword}

\end{frontmatter}


\section{Introduction}\label{s.first}
%%%%%%%%%%%%%%%%%%%%%%%%%%%%%%%%%%%%%%%%%%%%%%%%%%%%%%%%%%%%%%%%%%%%%%%%%%%%%
The research on geometric transformations for tetrahedra has its offspring in the importance of tetrahedral meshes in the computational engineering: physical properties of objects are virtually tested by simulations which discretize the object at the beginning, using mostly a tetrahedral or hexahedral mesh -- depending on the purpose. The convergence rate and error estimates of the numerical solution schemes depend severely on the mesh quality (see \cite{Fried1972,Babuska1976}): distorted mesh elements can decelerate or even impede the convergence of the numerical method and negatively affect the solution accuracy (see numerical examples in \cite{hughes2005}). These are the main reasons why at the beginning of every simulation, a lot of effort is put into the preprocessing in order to guarantee a high quality mesh. It should be remarked that tools for mesh smoothing exist in abundance using a wide range of mathematical methods (see overview in \cite{Frey2008}), but at the moment, they do not usually replace in totality the expensive correction of single elements by hand. \\
Looking for a simple, but effective smoothing method, \cite{VAGW08} introduced the geometric element mesh smoothing based on a generic geometric element transformation which relocates the vertices of every mesh element such that it becomes the most regular possible. Its good convergence properties and applicability were proven in a serial of articles (see \cite{VW12} and references within), also analytically for a single element in \cite{VH14b}. A very comprehensive presentation of this method is given in \cite[Section 6.3]{Lo2015} comparing its runtime, convergence and smoothing results to the dominating techniques in mesh smoothing.    \\
In this article, we study a different geometric transformation for tetrahedra which is derived from the rotational symmetry group of a regular tetrahedron. It can be the base of a geometric mesh smoothing algorithm analogous to the one mentioned above. In its motivational background, it is strongly related to the geometric triangle transformation introduced in \cite{VB14}. We call a transformation \emph{geometric} if it commutes with the similarity transformations of the euclidean space, these are translations, homotheties and rotations. Intuitively, the geometric transformations map similar tetrahedra to similar tetrahedra. \\
Our approach is a dynamical analysis of the tetrahedron transformation $\Theta$: where does a discrete orbit $\left\{\Theta^n(x)\;\big|\;n \geq 0\right\}$ for an arbitrary tetrahedron $x$ converge to? The analysis is based on two facts: firstly, we define the set of tetrahedra as the special linear group $\Sl(3,\mathbb{R})$ which is a non-compact Lie group; secondly, the transformation -- being geometric -- commutes with the rotational group $\So(3,\mathbb{R})$. Consequently, we can employ techniques from the theory of compact group actions and the geometric structure of Lie groups and their Lie algebras together with results of the theory of dynamical systems for the analysis of the dynamical behavior of the transformation. \\
For the understanding of dynamics, attractors play a central role: let $T$ be a continuous map on a topological space $X$, we then call a compact set $\mathcal{A}\subset X$ which is $T$-invariant, $T(\mathcal{A}) \subset \mathcal{A}$, an \emph{attractor} for $T$ if there exists an open $T$-invariant neighborhood $U \supset \mathcal{A}$ such that $\bigcap_{n \geq 0} T^n(U) = \mathcal{A}.$ The maximal set $U$ with the property above is called \emph{basin of attraction} because every $x \in U$ eventually converges under positive iterates of $T$ to $\mathcal{A}$. If $U=X$, the attractor is called a \emph{global attractor}. An attractor can have a complicated topological structure, e.g. being a fractal set, and carry a non-trivial dynamics $T|_{\mathcal{A}}$. \\  
The article is organized as following: we define the space of tetrahedra and discuss briefly its properties. In the subsequent section we motivate and introduce the tetrahedron transformation, the main subject of this article, and prove its basic properties.
We are then prepared to show the existence of a local attractor and the conditions under which its basin of attraction is the whole space of tetrahedra.  \\
While we focus in this article on a specific transformation, our methods are equally well applicable to study the dynamics of any $\So(3,\mathbb{R})$-equivariant transformation on $\Sl(3,\mathbb{R})$.  
 %%%%%%%%%%%%%%%%%%%%%%%%%%%%%%%%%%%%%%%%%%%%%%%%%%%%%%%%%%%%%%%%%%%%%%%%%%%% 
\section{Space of tetrahedra}
Every non-degenerate tetrahedron $x$ with vertices $x_0,x_1,x_2,x_3 \in \mathbb{R}^3$ can be written up to translation as a $(3,3)$-matrix $x=(x_1-x_0,x_2-x_0,x_3-x_0)$ with $\det(x)\neq 0$. Therefore, we consider the set $\Gl(3,\mathbb{R})$ of invertible real matrices which is a $9$-dimensional open affine subvariety in $\mathbb{R}^9$ with two connected components corresponding to matrices with negative and positive determinant, respectively. Tetrahedra which have the same shape, but different volumes, should not be distinguished. Consequently, as the volume of a tetrahedron $x$ can be computed by $\frac{1}{6}\det(x)$ it suffices to define the set of tetrahedra $X$ by $$X:= \left\{x \in \Gl(3,\mathbb{R}) \;\big|\; \det(x)=1\right\} =\Sl(3,\mathbb{R})$$ which is the \emph{special linear group} and a $8$-dimensional submanifold in $\mathbb{R}^9$.\\
Accordingly, every tetrahedron $x \in X$ can be written, by polar decomposition, as 
\begin{equation}\label{eq:polar}
x = \rho \sigma, \quad \rho \in \So(3,\mathbb{R}),\;\sigma\;\mbox{positive definite, symmetric matrix}.
\end{equation}
Positive definite symmetric matrices are diagonizable by an orthogonal basis, that is, there exists a orthogonal matrix $q$ such that $q^T\sigma q= \delta$ is a diagonal matrix. So, by orthogonal coordinate change and as the set of orthogonal matrices is a group we can express 
\begin{equation}\label{eq:polar2}
x':=q^Txq = q^T\rho q \delta=\rho' \delta, \mbox{where} \;\rho' \in O(3,\mathbb{R}).
\end{equation}
In this description, it becomes evident that a tetrahedron -- up to its volume and with respect to an appropriate orthogonal basis -- can be uniquely defined by the length of three of its edges (given in $\delta$) and the angles between them (defined in the three-dimensional rotational matrix $\rho'$).\\
\section{Derivation of a tetrahedron transformation}
\subsection{Motivation by rotational symmetry of regular tetrahedron}
Consider a rotation of angle $\frac{2\pi}{3}$ around an axis from a vertex of a regular tetrahedron to the barycenter of the opposite triangle. This transformation as a rotational symmetry sends the regular tetrahedron to itself. Let us extend that transformation to a bijection of the whole set of tetrahedra.\\
Consider a tetrahedron $T$ of $\mathbb{R}^3$ of vertices $x_0,\dots,x_3$ so that the vectors $(x_i-x_0)_{1\leq i\leq 3}$ form a positively oriented basis of $\mathbb{R}^3$. Let $\Delta_1$ be the face of vertices $x_1,x_2,x_3$. Define the \emph{pseudo rotation} $\theta_1$ of the triangle face $\Delta_1$ as follows: denoting by $c_1$ the centroid of $\Delta_1$, the points $\theta_{1}(x_i),i=1,2,3$ are the points in the plane $P_1$ spanned by $\Delta_1$ such that 
\begin{itemize}
\item the vector $\theta_{1}(x_2)-c_1$ is the image of the vector $x_1-c_1$ of the rotation of the plane $P_1$ by the angle $\angle x_1c_1x_2$ around the barycenter $c_1$,
\item the vector $\theta_{1}(x_3)-c_1$ is the image of the vector $x_2-c_1$ of the rotation of the plane $P_1$ by the angle $\angle x_2c_1x_3$ around the barycenter $c_1$,
\item the vector $\theta_{1}(x_1)-c_1$ is the image of the vector $x_3-c_1$ of the rotation of the plane $P_1$ by the angle $\angle x_3c_1x_1$ around the barycenter $c_1$.
\end{itemize}
Analogously, we define pseudo rotations for the remaining three triangle faces. We then have three different images of each vertex by the pseudo-rotations corresponding to the three neighboring faces. We define the final image of $T$ by taking their barycenters.\\
We give a complete formal definition in Section~\ref{sec:transformation}. See also Figure~\ref{fig:tetrahedron} for an illustration.
%The first counter-clockwise rotations by $\frac{2\pi}{3}$ can be generalized to an arbitrary non-regular tetrahedron by rotating every triangle face around its barycenter and mapping then every vertex onto the barycenter of its three images. We give a formal definition in Section~\ref{sec:transformation} and see Figure~\ref{fig:tetrahedron} for an illustration.  
\begin{figure}
\begin{center}
\includegraphics[width=0.5\textwidth]{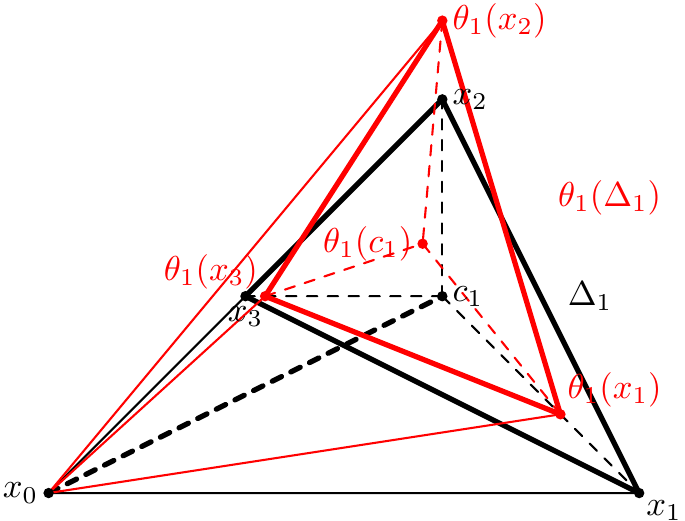}
\caption{The image of the \emph{pseudo-rotation} $\theta_1$ of the triangle face $\Delta_1$ is depicted in red. As the inner angles are not equal, every triangle vertex rotates by a different angle around the barycenter $c_1$ in the plane spanned by the triangle $\Delta_1$. Consequently, the barycenter of the triangle $\theta_1(\Delta_1)$ is usually not equal to $c_1$, in contrary to a real rotation. }
\label{fig:tetrahedron}
\end{center}
\end{figure}  
\subsection{Formal definition of the tetrahedron transformation}\label{sec:transformation}
%%%%%%%%%%%%%%%%%%%%%%%%%%%%%%%%%%%%%%%%%%%%%%%%%%%%%%%%%%%%%%%%%%%%%%%%%%%%%
Let $x=(x_1-x_0, x_2-x_0,x_3-x_0) \in \Gl(3,\mathbb{R})$ define a tetrahedron. Without loss of generality, we can assume that $x_0=0$. Denote the triangle faces by $\Delta_0=(x_0,x_1,x_3), \Delta_1=(x_1,x_2,x_3), \Delta_2=(x_2,x_0,x_3),\Delta_3=(x_0,x_1,x_2)$. As described above, we define first triangle transformations $\theta_i:\mathbb{R}^{9} \rightarrow \mathbb{R}^9,i=0,\dots,3$ for the triangle faces $\Delta_i$ which express the pseudo-rotations (see also the discussion of this triangle transformation in \cite{VB14}). Let $c_i$ denote the corresponding barycenters of the triangle faces $\Delta_i$, we then define for $j=0,1,2$:
\begin{align*}
\theta_{ij}(\Delta_i)&=\left\|x_{i_{j-1}}-c_i\right\|\left\|x_{i_j}-c_i\right\|^{-1}(x_{i_j} - c_i) + c_i, \; \Delta_i=(x_{i_0},x_{i_1},x_{i_2}).
\end{align*}
The tetrahedron transformation is then the result of \emph{rotating} the triangle faces and mapping each vertex onto the barycenter of the three images of the triangle face transformations. The vertex $x_0=0$ is eventually pushed back into the origin. Formally, the tetrahedron transformation is then given by  
\begin{align*}
&\tilde{\Theta}: (\mathbb{R}^3)^3 \rightarrow (\mathbb{R}^3)^3\\
&x=(x_1-x_0,x_2-x_0,x_3-x_0) \mapsto \left(\tilde{\Theta}_0(x), \tilde{\Theta}_1(x),\tilde{\Theta}_2(x)\right)\\
&\mapsto \frac{1}{3}\begin{pmatrix}
\theta_{01}(\Delta_0)+\theta_{10}(\Delta_1) + \theta_{31}(\Delta_3) - (\theta_{00}(\Delta_0) + \theta_{21}(\Delta_2) + \theta_{30}(\Delta_3))\\
\theta_{11}(\Delta_1)+\theta_{20}(\Delta_2) + \theta_{32}(\Delta_3) - (\theta_{00}(\Delta_0) + \theta_{21}(\Delta_2) + \theta_{30}(\Delta_3))\\
\theta_{02}(\Delta_0)+\theta_{12}(\Delta_1) + \theta_{22}(\Delta_2) - (\theta_{00}(\Delta_0) + \theta_{21}(\Delta_2) + \theta_{30}(\Delta_3))\end{pmatrix},
\end{align*}
where $\theta_i$ denotes the geometric triangle transformation of every triangle face $\Delta_i, i=0,\dots,3$ with centroid $c_i$. The transformation $\tilde{\Theta}$ can be restricted to $\Sl(3,\mathbb{R})$. Let $\lambda.x$ denote the multiplication of each column vector of $x$ by $\lambda \in \mathbb{R}$. One easily computes that 
$$\lambda.\tilde{\Theta}(x) = \tilde{\Theta}(\lambda.x),$$
that is, the transformation commutes with homotheties. We therefore define for $\tilde{\Theta}(x) \in \Gl(3,\mathbb{R})$
\begin{align*}
a(x)&=\left(\frac{\vol(x)}{\vol(\tilde{\Theta}(x))}\right)^{\frac{1}{3}}=\left(\frac{\det(x)}{\det(\tilde{\Theta}(x))}\right)^{\frac{1}{3}}.
\end{align*}
and set for all $x \in \Gl(3,\mathbb{R})$ with $\tilde{\Theta}(x) \in \Gl(3,\mathbb{R})$:   
$$\Theta(x) := a(x).\tilde{\Theta}(x)$$
so that $\det(\Theta(x))=\det(x)$. \\
The goal of this article is to study the dynamical system defined on the set of tetrahedra by setting 
$x^{(n)} = \Theta(x^{(n-1)})$
for any tetrahedron $x^{(0)}$ for $n \geq 1$. In particular, we are interested in its limit behavior 
$$\lim_{n \rightarrow \infty} \Theta(x^{(n)}) $$
for arbitrary $x^{(0)}$ inside an appropriately defined subset $X \subset (\mathbb{R}^{3})^3$ of tetrahedra. 
\subsection{Properties of the transformation $\Theta$} 
The transformation $\Theta$ can be naturally restricted to the space $\Sl(3,\mathbb{R})$. In this way, we focus our considerations upon the space of tetrahedra $X=\Sl(3,\mathbb{R})$. Momentarily, we ignore if $\Theta(\Sl(3,\mathbb{R})) \subset \Sl(3,\mathbb{R})$, so the following lemma has to be shown at the very beginning:
\subsubsection{Well-definedness of $\Theta$}
\begin{lemma}
For any $x \in \Sl(3,\mathbb{R})$ we have $\Theta(x)\in \Sl(3,\mathbb{R})$. Further, there exists $U \subset X$ such that the map $\Theta: U \subset X \rightarrow X$ is well-defined.
\end{lemma}
\begin{proof} 
First of all, the triangle transformations $\theta_{i}$ map linearly independent vectors onto linearly independent vectors. Consequently, $\tilde{\Theta}$ maps $\Gl(3,\mathbb{R})$ into $\Gl(3,\mathbb{R})$. By the renormalization of the determinant, the image $\Theta(x)$ lies always in $X=\Sl(3,\mathbb{R})$. \\
We compute $\Theta(x_{eq})=x_{eq} \in X$, for $x_{eq}$ a regular tetrahedron. The map $\Theta$ is smooth, and one computes that $D\Theta(x_{eq})$ has rank $8$, so $x_{eq}$ is a regular point. Due to the implicit function theorem, there exists an open set $U\subset \Sl(3,\mathbb{R})$ around $x_{eq}$ such that $\Theta|_U$ is a diffeomorphism and consequently well-defined on $U$. More generally, as $\Theta$ is differentiable on $X$, the set of singular points is a meager set, so we find a dense and open set $U \subset X$ such that $\Theta|_U$ is a diffeomorphism. 
\end{proof}
\subsubsection{Equivariance of $\Theta$}
Now, we define a group action of the group $G:=\So(3,\mathbb{R})$ on $\Sl(3,\mathbb{R})$ by matrix multiplication from the left:
$$(\rho,x) \in \So(3,\mathbb{R}) \times \Sl(3,\mathbb{R}) \mapsto \rho x \in \Sl(3,\mathbb{R}).$$
\begin{lemma}
The group $G$ as defined above acts freely and properly discontinuously on $\Sl(3,\mathbb{R})$.
\end{lemma}
\begin{proof}
Let $\rho \in G$ and $x \in \Sl(3,\mathbb{R})$. Then $\rho x \in\Sl(3,\mathbb{R})$ as $\det(\rho)=1$. Assume now that $\rho x=x$. This implies that every column of $x$ has to be an eigenvector to the eigenvalue $1$ of $\rho$, so we can conclude that $\rho$ is  the identity $\id$. The preimages of compact sets in $\Sl(3,\mathbb{R}) \times \Sl(3,\mathbb{R})$ under the map $(\rho,x) \mapsto (x,\rho x)$ are compact as a rotation is an isometry, so the group acts properly discontinuously. 
\end{proof}
This gives us directly the following Corollary (see e.g.\cite[(5.2) Proposition]{dieck1987}): 
\begin{corollary}\label{cor:equivariance}
The orbits $\So(3,\mathbb{R})x=\left\{\rho x\right\}_{\rho \in \So(3,\mathbb{R})}$ are three-dimensional smooth manifolds, all homeomorphic to $\mathbb{R}P^3$. The orbit space $X/\So(3,\mathbb{R})$ is a non-compact homogeneous manifold of dimension $5$.   
\end{corollary}   
\begin{rem}
\begin{enumerate}
\item In particular, the orbit $\So(3,\mathbb{R})x_{eq}$ of a regular tetrahedron is a $3$-dimensional compact submanifold of $X$.
\item Note that the subgroup $\So(3,\mathbb{R}) < \Sl(3,\mathbb{R})$ is not a normal subgroup, so the orbit space is not a Lie group itself, but only a smooth manifold. 
\item We define an equivalence relation as follows: For all $x,y \in X$, $x \sim y$ iff $y \in Gx$. This gives us a canonical projection $p$ onto the quotient space of symmetric positive definite matrices with determinant $1$ by $p: \Sl(3,\mathbb{R}) \rightarrow \Sl(3,\mathbb{R})/\So(3,\mathbb{R}),\quad p: x \mapsto [x]_{Gx}$ and the following commutative diagram. 
\begin{equation*}
  \xymatrix{
\Sl(3,\mathbb{R}) \ar[d]^p \ar[r]^{\Theta} &\Sl(3,\mathbb{R})\ar[d]^p\\
\Sl(3,\mathbb{R})/\So(3,\mathbb{R}) \ar[r]^{\Theta} &\Sl(3,\mathbb{R})/\So(3,\mathbb{R})} \quad \xymatrix{
x \ar[d]^p \ar[r]^{\Theta} &\Theta(x)\ar[d]^p\\
Gx \ar[r]^{} &\Theta(Gx)=G\Theta(x) }
 \end{equation*}
It suffices therefore to consider tetrahedra $x \in \Sl(3,\mathbb{R})/\So(3,\mathbb{R})$.  
\end{enumerate}
\end{rem}
\subsubsection{Properties of the space $\Sl(3,\mathbb{R})/\So(3,\mathbb{R})$}\label{sec:properties}
Let $X:=\Sl(3,\mathbb{R}), G:=\So(3,\mathbb{R})$. Let us summarize some essential properties of the quotient space $X/G$ which are taken from \cite[Chapter 5]{Jost} and \cite{Hegalson} where symmetric spaces are extensively discussed. The Lie algebra $\mathfrak{sl}=T_{\id}X$ of $X$ consists of matrices with vanishing trace while the Lie algebra $\mathfrak{g}=T_{\id}G$ consists of skew-symmetric matrices $a$ with $a=-a^T$. Let $\mathfrak{p}$ be the Lie algebra of symmetric matrices $a$ with $a^T=a$ with vanishing trace. One can then show that there exists a \emph{Cartan decomposition} of $X$, that is $\mathfrak{sl}=\mathfrak{g} \oplus \mathfrak{p}$. Using the exponential, one derives that the space $X/G$ is homeomorphic to the set of symmetric positive definite matrices $P$ and, inheriting the differential structure of $P$, a differential manifold. Hence, one can identify the quotient space $X/G$ with the set of symmetric positive definite matrices with determinant $1$. \\
As $X/G$ is simply connected, one computes the tangent space for any $x \in X/G$ by $xT_{\id}(X/G)=x\mathfrak{p}$. As basis of $\mathfrak{sl}$ one chooses $E^{ij}, i\neq j,i,j=1,\dots,3$ where $E^{ij}$ are matrices with zero entries except from one entry equal to $1$ at the position $ij$ together with $H^i=\frac{1}{\sqrt{2}}(E^{ii}-E^{i+1,i+1})$ for $i=1,2$. This gives us directly the basis of $\mathfrak{p}=\left\{a \in \mathfrak{sl}\big|\; a^T=a\right\}$ by $H^1,H^2, \frac{1}{\sqrt{2}}(E^{12}+E^{21}), \frac{1}{\sqrt{2}}(E^{13}+E^{31})$ and $\frac{1}{\sqrt{2}}(E^{23}+E^{32})$. In this way, we have at any point $x \in X/G$ a canonical basis on its tangent space.  
\section{Existence of local attractor}
\subsection{Set of regular tetrahedra}
As seen before, the regular tetrahedron $x_{eq}$ is a fixed point of $\Theta$. One further shows the following: 
\begin{lemma}\label{lemma:attractor}
let $\mathcal{A}$ denote the orbit $Gx_{eq}$ of a regular tetrahedron $x_{eq} \in \Sl(3,\mathbb{R})$. Then $\mathcal{A}$ is a $3$-dimensional local attractor of $\Theta$. Moreover, the following properties hold:
\begin{enumerate}
\item $\Theta(\mathcal{A})=\mathcal{A}$,
\item there exists an open neighborhood $U \subset \Sl(3,\mathbb{R})$ of $\mathcal{A}$ such that for any $x \in U$ we have
$$\lim_{n\rightarrow \infty} \Theta^n(x) =x^{*} \in \mathcal{A}.$$
\end{enumerate} 
\end{lemma} 
\begin{proof}[Proof of Lemma~\ref{lemma:attractor}]
First of all, we show a preliminary result for the Jacobian of $\Theta$: 
\begin{pro}\label{pro:jacobian}
Let $x \in \Sl(3,\mathbb{R})$, then the Jacobian $D\Theta(x)$ of $\Theta$ commutes with orthogonal matrices, that is
$$D\Theta(\rho x) = \rho^TD\Theta(x)\rho$$
where $\rho \in \So(3,\mathbb{R}).$
\end{pro}
\begin{proof}[Proof of Proposition~\ref{pro:jacobian}]
The proof is by computation, applying the derivation rule for concatenated transformations and utilizing that $\rho\Theta(x)=\Theta(\rho x)$. Deriving the left hand side, we get
\begin{equation}\label{eq:first}
D(\rho \circ \Theta(x)) = D(\Theta(x))|_x \circ D(\rho(\Theta(x))|_{\Theta(x)} = D(\Theta(x))|_x \circ \rho.
\end{equation}
For the right hand side, we get
\begin{equation}\label{eq:right}
D(\Theta \circ \rho(x)) = D(\rho(x))|_x \circ D(\Theta(\rho(x)))|_{\rho x} = \rho \circ D(\Theta(\rho x))|_{\rho x}.
\end{equation}
This results in 
\begin{equation}\label{eq:final}
D(\Theta(x))|_x \circ \rho = \rho \circ D(\Theta(\rho x))|_{\rho x}
\end{equation}
finishing the proof.
\end{proof}
Consider now an arbitrary regular tetrahedron $x_{eq} \in \Sl(3,\mathbb{R})$. 
Numerically, we compute the Jacobian $D\Theta(x_{eq})$ and its eigenvalues: the Jacobian has three eigenvalues with absolute value $1$ corresponding to the three-dimensional tangent space to the group orbit $\mathcal{A}$. The remaining five eigenvalues $\lambda_4,\dots,\lambda_8$ of the Jacobian consist of one pair of complex conjugate and three real eigenvalues which have all an absolute value strictly smaller than $1$. As the eigenvalues are an invariant for similar matrices, Equation~\ref{eq:final} implies that the eigenvalues of the Jacobian for any $x \in \mathcal{A}$ are exactly the same. Consequently, the tangent space at $x \in \mathcal{A}$ splits in two $D\Theta$-invariant subspaces
$$T_{x}\Sl(3,\mathbb{R}) = T_{x}\mathcal{A} \oplus E^s_{x},\quad x \in \mathcal{A}.$$
The tangent space $T_x\mathcal{A}$ is spanned by three skew-symmetric matrices which is in accordance to the fact that the group of skew-symmetric matrices is exactly the Lie algebra of $\So(3,\mathbb{R})$. The subspace $E^s_{x}$ is spanned by the five eigenvectors corresponding to $\lambda_4,\dots,\lambda_8$. There exists therefore $c < 1$ such that 
$$\left\|D\Theta(x)v\right\| < c\left\|v\right\|, \quad v \in E^s_{x}, x \in \mathcal{A}.$$
The numerical computation gives a value $c \approx 0.68$. 
Therefore, by the Stable Manifold Theorem (see \cite[Lemma 6.2.7]{katok1995}), there exists for any $x \in \mathcal{A}$ locally a five-dimensional disc $W^s(x)$ with the following properties: 
\begin{enumerate}
\item $x \in W^s(x)$ 
\item $T_{x}W^s(x) \subset E^s_{x}$
\item $\left\|\Theta^n(y)-x\right\| \rightarrow 0,\; \mbox{for any}\; y \in W^s(x).$
\end{enumerate}
The orbit $\mathcal{A}$ of the regular tetrahedron is therefore a local attractor. 
\end{proof}
\begin{rem}
In fact, we proved a stronger result: The positive iterates $\Theta^n(x)$ for tetrahedra sufficiently close to $\mathcal{A}$ do not only converge to $\mathcal{A}$, but to a single regular tetrahedron $x^{*} \in \mathcal{A}$ as $\mathcal{A}$ is not only a $\Theta$-invariant set, but a set of fixed points.
\end{rem}
\begin{figure}
\begin{center}
\includegraphics[width=0.25\textwidth]{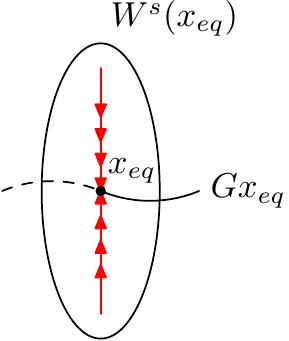}
\caption{The stable disk of $x_{eq} \in \mathcal{A}$ with iterates $\Theta^n(x), x \in W^s(x_{eq})$ converging to $x_{eq}$ in red.}
\label{fig:stableset}
\end{center}
\end{figure}
\subsection{Existence of further fixed points}
The natural question is if any non-regular tetrahedron $x \in \Sl(3,\mathbb{R})\setminus \mathcal{A}$ is a fixed point of $\Theta$. Let $x$ be any tetrahedron. The iteration step of the Newton method is defined for $x=x^{(0)}$ and $n\geq 0$ by
\begin{align*}
(D\Theta(x^{(n)})-\id)y &=-\Theta(x^{(n)})+x^{(n)},\\
x^{(n+1)} &= y + x^{(n)}.
\end{align*}
We performed numerical tests with 1000 randomly chosen tetrahedra and the Newton iteration converged for all of them to a regular tetrahedron which does not prove, but strongly suggests that the only fixed points of $\Theta$ are indeed the regular tetrahedra:
\begin{statement}\label{statement}
The numerical search for fixed points of $\Theta$ only determine regular tetrahedra as fixed points.
\end{statement}  
\section{From local towards global dynamics}
First of all, we would like to stress that $\Theta$ does not commute with right multiplication by orthogonal matrices, that is, we have in general 
$$\Theta(x\rho)\neq\Theta(x)\rho.$$
But this implies that we cannot perform an orthogonal coordinate change such that $x=\rho \delta$ is the product of an orthogonal and a diagonal matrix. \\
In order to prove the global convergence of an arbitrary tetrahedron $x$ towards a regular tetrahedron we have to look at the dynamics along an orbit $\Theta^n(x), n\geq 0$ and show that there exists $N < \infty$ such that $\Theta^N$ on $X/G$ is a contraction.\\
The space $X=\Sl(3,\mathbb{R})$ is $8$-dimensional and the group orbits $\left\{Gx=\So(3,\mathbb{R})x\right\}_{x \in X}$ decompose $X$ into compact $3$-dimensional manifolds. At every $x \in X$ the tangent space $T_xX$ splits into the tangent space of the group orbit $T_xGx$ and its orthogonal complement $E^s_x$, that is $T_xX=T_xGx \oplus E^s_x$. The tangent space of the group orbit is certainly $D\Theta$-invariant, that is $D\Theta(x)(T_xGx) \subset T_{\Theta(x)}G\Theta(x)$, while the complement $E^s_x$ is generally not. 
\begin{figure}
\begin{center}
\includegraphics[width=0.7\textwidth]{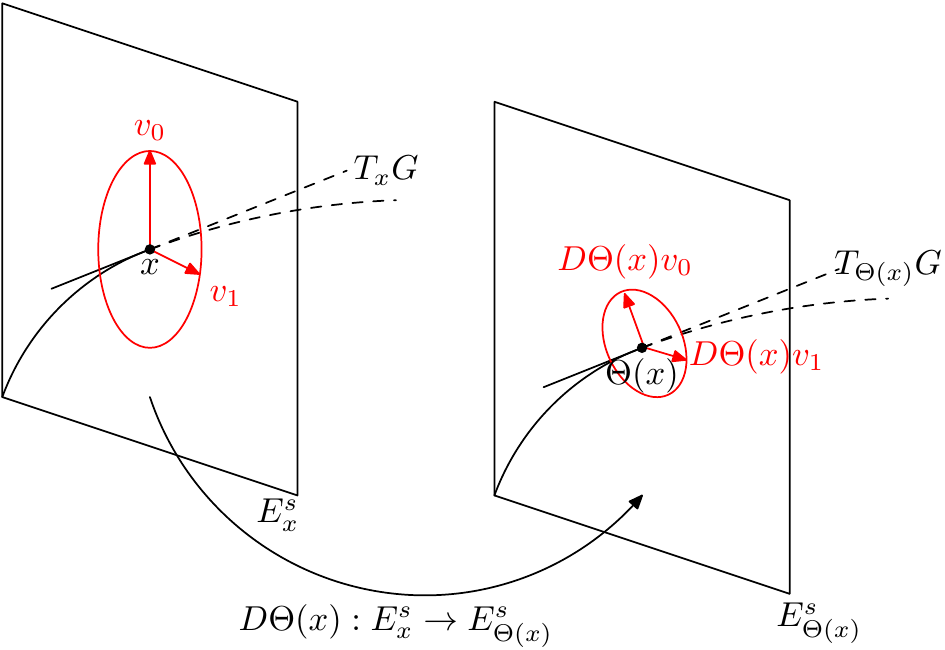}
\caption{Dynamics along an orbit of $\Theta$: the dynamics of $\Theta$ is controlled by the stretching and contracting inside the tangent subspace $E^s_x$ orthogonal to the tangent space of the group orbit $Gx$. }
\label{fig:dynamic}
\end{center}
\end{figure}
\begin{lemma}\label{lemma:contracting}
Let $x \in X$ and let $T_xX=T_xGx \oplus E^s_x$ be the splitting of the tangent space $T_xX$ at $x$. Assume that  $$\max_{v \in E^s_x}\left\|D\Theta(x)v\right\| < \left\|v\right\|,$$ 
then there exists $\epsilon > 0$ such that 
\begin{enumerate}
\item $D_{\epsilon,x} \times Gx$ is a tubular neighborhood of $Gx$, where $D_{\epsilon,x}$ is a disk centered at $x$ of radius $\epsilon >0$ transverse to the group orbit;
\item there exists $0 < \lambda < 1$ such that $\Theta(D_{\epsilon,x} \times Gx) \subset D_{\lambda \epsilon,\Theta(x)} \times G\Theta(x)$.
\end{enumerate}
\end{lemma}
\begin{proof}
First of all, one easily deduces that the assumption holds equally for all points $y=\rho x \in Gx$ as $D\Theta(y)=\rho^TD\Theta(x)\rho$. \\
As $G$ acts freely and properly discontinuously, the Slice Theorem for compact Lie groups is applicable (see \cite{Koszul1953},\cite{Palais1957}) which says that there exists a tubular neighborhood $V$ of $Gx$ with a trivial product structure. More precisely, there exists $\eta > 0$ such that $V$ diffeomorphic to the product $Gx \times D_{\eta,x}$ where $D_{\eta,x}$ is a disk around $x$ such that $T_xD_{\eta,x} \subset E^s_x$. Further, as $\Theta$ is $G$-equivariant and smooth, $\Theta(V)$ is as well a tubular neighborhood of $G\Theta(x)$, and in particular, we have $\Theta(D_{\eta,x}) \subset \Theta(V)$. The tangent space $E^s_x$ may not be $D\Theta$-invariant, so $D\Theta(E^s_x) \not \subset E^s_{\Theta(x)}$.\\
Firstly, due to the Lie group structure we can slide $D\Theta(x)E^s_x$ along $G\Theta(x)$ by rotations such that $D\Theta(x)E^s_x \cap T_{\Theta(x)}G\Theta(x)=\left\{\Theta(x)\right\}$. This does not change the vector lenghts as rotations are isometries. Let then $v \in D\Theta(x)E^s_x$. We can project orthogonally $D\Theta(x)v$ onto $E^s_{\Theta(x)}$, and we have $\left\|\pr D\Theta(x)v\right\| \leq \left\|D\Theta(x)v\right\|$. Accordingly, if $D\Theta(x):E^s_x \rightarrow D\Theta(x)E^s_x$ is contracting, then $\pr \circ D\Theta(x):E^s_x \rightarrow E^s_{\Theta(x)}$ is also contracting. Repeating this procedure for any point $y \in Gx$ and by the compactness of $Gx$, we can therefore conclude that there exists $\epsilon > 0$ such that $\Theta(D_{\epsilon,x}) \times Gx) \subset D_{\lambda\epsilon,\Theta(x)} \times G\Theta(x)$ for $\lambda < 1$.
\end{proof} 
\begin{example}
We consider the dynamics of $\id$. The maximal dilation \newline$\max_{v \in E^s_{\id},\left\|v\right\|} \left\|D\Theta(\id)v\right\|$ is exactly the singular value corresponding to directions transverse to the group orbit $G\id$. The singular value decomposition of $D\Theta(\id)=U\Sigma V$ gives us approximate singular values $1.191$, $1.146$, $0.897$, $0.853$, $0.767$, $0.477$, $0.417$. The tangent space of the group orbit is spanned by skew-symmetric matrices. So we have to look which of the singular values correspond to directions tangent to the group orbit. Evaluating the computed orthogonal basis $V$ of $T_{\id}X$, we observe that the first three columns are skew-symmetric matrices so the maximal singular value transverse to the group orbit is $0.853$. This means, that $D\Theta(\id)|_{E^s_{\id}}$ is contracting.   
\end{example}
To prove the convergence of $x$ to a fixed point orbit one needs therefore that there exists a sequence $n_k\rightarrow_{k\rightarrow \infty} \infty $ such that $\max_{v \in E^s_x, \left\|v\right\|=1}\left\|D\Theta^{n_k}(x)v\right\| $ for $k \rightarrow \infty$ is a strictly monotonically decreasing sequence.\\
The maximal dilation $\max_{v \in E^s_{x}, \left\|v\right\|=1}\left\|D\Theta^{n_k}(x)v\right\|$ is equal to the maximal singular value $\sigma(x,n_k)$ of the matrix $D\Theta^{n_k}(x)$ transverse to the group orbit $G\Theta^{n_k}x$. We have to show therefore that 
\begin{equation}\label{eq:lyapunov}
\limsup_{n \geq 0} (\sigma(x,n))=0.
\end{equation}
\begin{lemma}\label{lemma:global}
Let $x \in X$ with the property given by Equation~\ref{eq:lyapunov}.
Then there exists $x^*$ such that $\lim_{n \rightarrow \infty}\Theta^n(x)=x^*.$
\end{lemma}
\begin{proof}
Let $x \in X$ with the property above. Then there exists a subsequence $\sigma(n_k,x)\rightarrow_{k\rightarrow \infty} 0$. Utilizing Lemma~\ref{lemma:contracting} consecutively to the transformation $\Theta^{n_{k}}$, this gives us a strictly monotonically decreasing sequence of diameter $\epsilon_k > 0$ such that 
$$\bigcap_{k \geq 0}\Theta^{n_k}(D^s_{\epsilon,x} \times Gx) \subset \bigcap_{k \geq 0} D^s_{\epsilon_k,\Theta^{n_k}(x)} \times G\Theta^{n_k}x = G \left(\bigcap_{k \geq 0}\Theta^{n_k}(x)\right).$$
This intersection can only consist of one group orbit due to the construction. Consequently, there exists $x^* \in X$ such that $\bigcap_{k \geq 0}\Theta^{n_k}(x)=:x^*$ proving the assertion.
\end{proof} 
\subsubsection{Numerical verification}
As the transformation itself cannot be treated analytically, we can compute the numbers $\limsup \sigma(x,n)$ only for a finite number of points $x \in X$. More precisely, for any arbitrary $x \in \Sl(3,\mathbb{R})$ we take the basis $x\mathfrak{sl}(3,\mathbb{R})$ of $T_xX$ given by the basis $H^1,H^2,E^{12},E^{21},E^{13},E^{31},E^{23},E^{32}$ introduced in Section~\ref{sec:properties} multiplied by $x$. 
We compute $D\Theta^n(x), n\geq 0$, by numerically deriving $\Theta^n(x)$ into the directions of the basis vectors and set $\sigma_i(x,n)$ equal to the $i$-th singular value of $D\Theta^n(x)$. We increase $n$ until $\left\|\sigma(x,n)-\sigma(x,n+1)\right\| < 10^{-8}$ and take then the resulting vector $\sigma(x,n)$ as $\limsup_{n \geq 0}\sigma(x,n)$. Figure~\ref{fig:singular} shows the resulting vector for more than 1000 test tetrahedra.
\begin{figure}[h]
\begin{minipage}{0.49\textwidth}
\includegraphics[width=\textwidth]{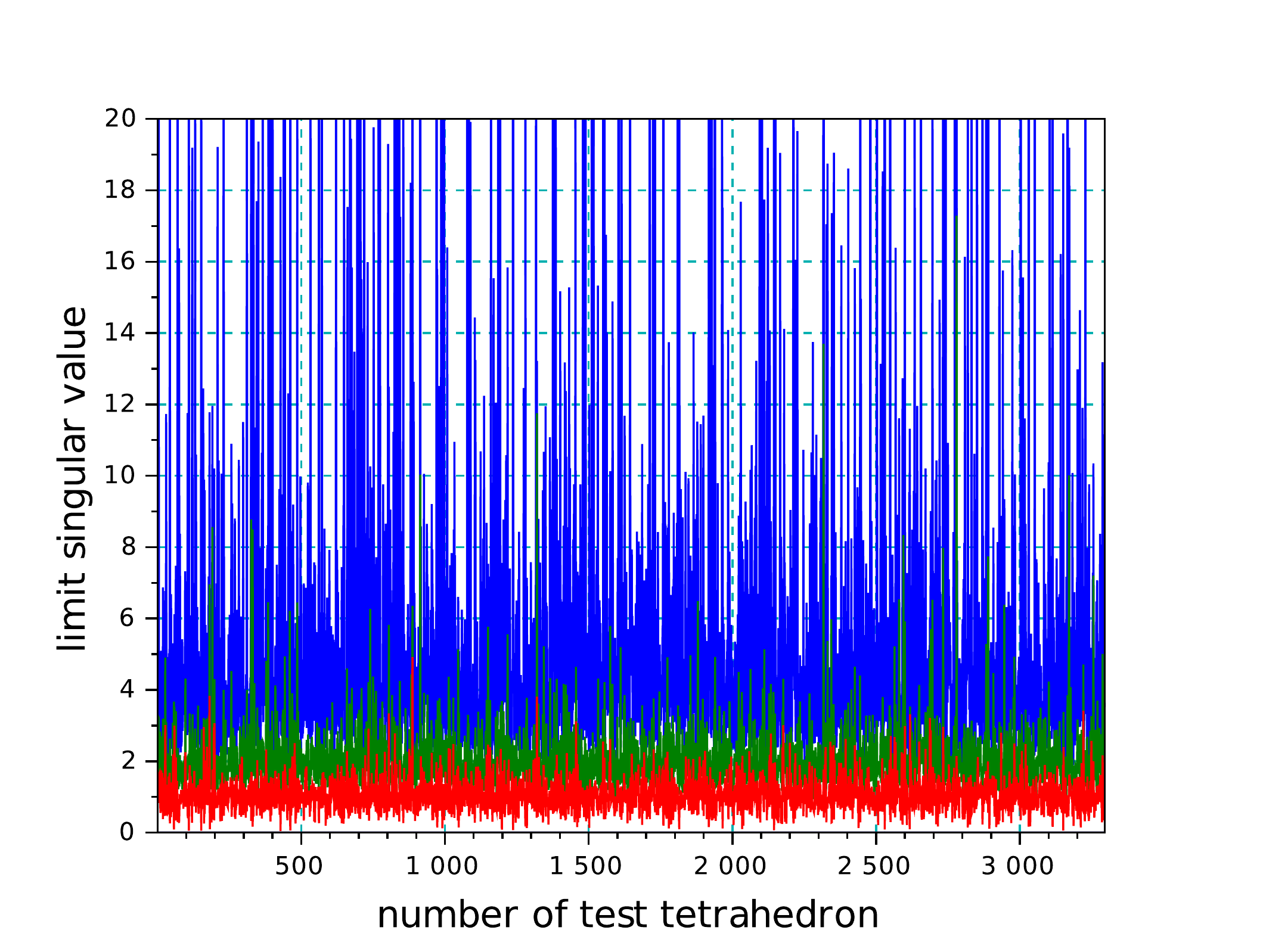}
\end{minipage}
\begin{minipage}{0.49\textwidth}
\includegraphics[width=\textwidth]{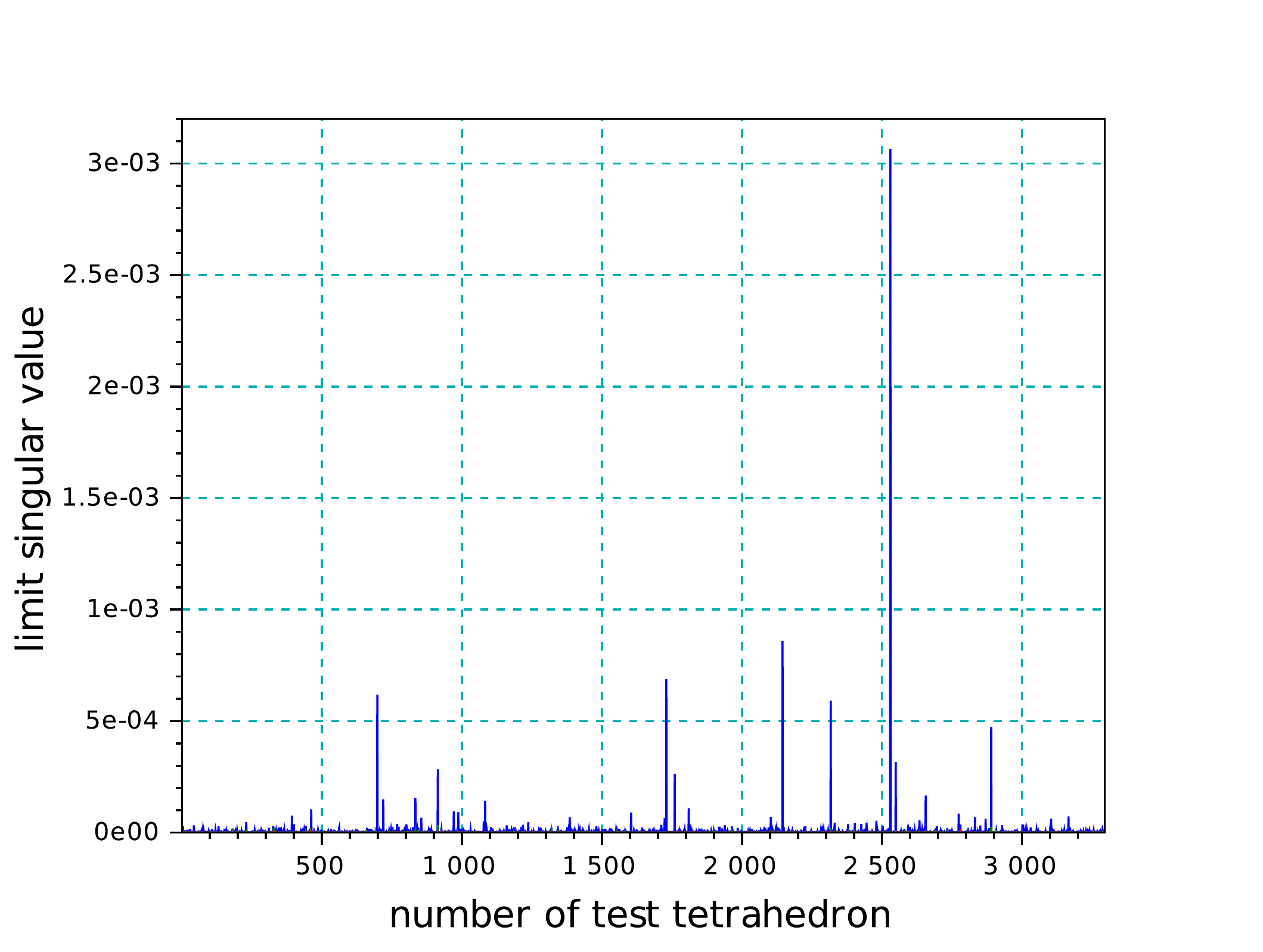}
\end{minipage}
\caption{On the left, the eight limit singular values $\lim \sigma(x,n)$ for test tetrahedra $x$ are shown. One observes that there are exactly three singular values greater than $1$ corresponding to the three directions tangent to the group orbit. The other limit singular values vanish as it can be clearly seen on the right where these five singular values corresponding to the orthogonal complement to $TG$ are exclusively displayed.}
\label{fig:singular}
\end{figure}
\begin{statement}
Numerical computations of the limit singular values \linebreak$\limsup_{n \geq 0}\sigma(x_i,n)$ for randomly chosen test tetrahedra $x_i \in \Sl(3,\mathbb{R}), i=1,\dots,1400$ give us a typical singular value distribution as displayed in Figure~\ref{fig:singular}: three singular values greater than $1$, the others, corresponding to the directions transverse to the group orbit, converge to zero. This suggests that the hypothesis of Lemma~\ref{lemma:global} is fulfilled for $x \in \Sl(3,\mathbb{R})$. Together with Statement~\ref{statement} this implies that the set of regular tetrahedra is a global attractor of $\Theta$.
\end{statement} 
\section{Outlook and Conclusions}
Although we studied one specific tetrahedron transformation, we are convinced that the methods presented above can be useful to understand the dynamics of any tetrahedron transformation thanks to the geometric structure of homogeneous spaces. Therefore, the key for a thorough understanding was certainly the appropriate definition of the domain of our transformation, the space of tetrahedra, as the homogeneous space $\Sl(3,\mathbb{R})/\So(3,\mathbb{R})$. \\
Two natural generalizations of our approach lend themselves:  
\begin{description}
\item[Platonic bodies] Is this approach generalizable to other platonic bodies, notably hexahedra, octahedra, dodecahedra and icosahedra? Their symmetry group are certainly more complicated and the space of hexahedra is not as directly describable than the one of tetrahedra, too. But maybe, there is an elegant way to apply our dynamical approach to geometric transformation of these bodies, too. 
\item[Meshes] A volumetric mesh is defined by a set of points in the euclidean space and the so-called topology which describe which of these points in which order are connected to an $n$-body. Meshes usually arise from the discretization of space for the numerical computation of partially differential equations. As the tetrahedron transformation above was invented and is employed to improve the quality of tetrahedral meshes, it is natural to ask if one could in fact prove that the tetrahedral transformation improves the mesh quality of any tetrahedral mesh (or of a class of tetrahedral meshes). In \cite{VB14} we have already studied the numerical properties of the derived mesh transformation, but a mathematical analysis is still to be accomplished. 
\end{description}

\bibliographystyle{plain} 
\bibliography{bibfile}




\end{document}